\documentclass[12pt, oneside,reqno]{amsart}
\usepackage{amsmath, amsfonts, amssymb}
\usepackage{eucal}
\usepackage{mathrsfs}
\usepackage{latexsym}
\usepackage{enumitem}
\usepackage{bbm}
\usepackage{mathtools}
\usepackage{exscale}
\usepackage{cases}

\usepackage[pdfstartview=FitH,
            CJKbookmarks=true,
            bookmarksnumbered=true,
            bookmarksopen=true,
            colorlinks=true,
            linkcolor=blue,
            anchorcolor=blue,
            citecolor=red,
            urlcolor=blue
            ]{hyperref}

\numberwithin{equation}{section}
\newtheorem{theorem}{Theorem}[section]
\newtheorem{definition}[theorem]{Definition}
\newtheorem{lemma}[theorem]{Lemma}
\newtheorem{corollary}[theorem]{Corollary}

\newtheorem{proposition}[theorem]{Proposition}

\newtheorem{definition and theorem}[theorem]{Definition and Theorem}
\def\bl{\begin{lemma}}
\def\el{\end{lemma}}
\def\bc{\begin{corollary}}
\def\ec{\end{corollary}}
\def\bt{\begin{theorem}}
\def\et{\end{theorem}}

\newcommand{\tA}{\mathcal{A}_s}

\def\bp{\begin{proposition}}
\def\ep{\end{proposition}}
\def\be{\begin{equation}}
\def\ee{\end{equation}}
\def\baa{\begin{align*}}
\def\eaa{\end{align*}}

\def\bd{\begin{definition}}
\def\ed{\end{definition}}

\allowdisplaybreaks[4]
\theoremstyle{plain}

\theoremstyle{remark}
\newcommand{\lsub}[1]{\hskip -1.0pt\lower.3ex\hbox{$_{#1}$}}

\theoremstyle{definition}

\newcommand{\sn}{{\mathbb S}^{n-1}}

\newcommand{\tr}{\mathbb R}
\newcommand{\rn}{\mathbb R^n}

\newcommand{\tk}{\mathcal{K}}

\newcommand{\tH}{\mathcal{H}}

\title[Anisotropic fractional area measures]{Anisotropic fractional area measures}

\author{Xiaxing Cai}
\begin{document}
\date{}

\begin{abstract}
    The anisotropic $s$-fractional area measures are introduced as the first variation of the anisotropic fractional $s$-perimeter $P_s(K,L)$, with $L$ an origin symmetric convex body and $s\in(0,1)$. As $s\rightarrow 1^-$, the anisotropic  $s$-fractional area measure converges to the mixed area measure of $K$ and the moment body of $L$. The Minkowski problem of these measures are solved. Finally, a necessary condition for the convexity of optimizers in the anisotropic fractional isoperimetric inequality is derived.
\end{abstract}

\maketitle

\section{Introduction}

%%%%先介绍anisotropic fractional perimeter及其等周不等式,以及收敛性等结果

For a bounded Borel set $E\subset\rn$, an origin symmetric convex body $L$ and $s\in(0,1)$, Ludwig \cite{Lud1} defined the anisotropic fractional $s$-perimeter of $E$ with respect to $L$ by
\[P_s(E,L)=\int_{E}\int_{E^c}\frac{1}{\|x-y\|_L^{n+s}}dxdy,\]
where $E^c$ is the complement of $E$ in $\rn$ and $\|\cdot\|_L$ is the norm with closed unit ball $L$. Here we say $L\subset \rn$ is a convex body if $L$ is a convex compact set with nonempty interior. 
The anisotropic fractional isoperimetric inequality states that there is an optimal constant $\gamma_s(L)>0$ such that
\[P_s(E,L)\ge \gamma_s(L)|E|^{\frac{n-s}{n}},\]
where $|E|$ denotes the Lebesgue measure of $E$.

The classic fractional $s$-perimeter of $E$, denoted by $P_s(E)$, is $P_s(E,B^n)$ where $B^n$ is the Euclidean unit ball in $\rn$. The fractional isoperimetric inequality is
\[P_s(E)\ge \gamma_{n,s}|E|^{\frac{n-s}{n}},\]
where the equality is attained precisely by balls, up to a null set (see \cite{FS}). When $L\neq B^n$, little is known about $\gamma_s(L)$ and the optimizer of the anisotropic fractional isoperimetric inequality. Ludwig \cite{Lud1} considered the limiting behavior of $P_s(E,L)$ as $s\rightarrow 1^-$, which showed that the optimizers are not homothetic to the unit ball. By a result of Kreuml \cite{Kr}, if $L$ is unconditional, the optimizer is an unconditional star body. 
The anisotropic fractional perimeter is closely related to the theory of fractional Sobolev norms and fractional isoperimetric inequalities, which have attracted considerable attention over the past three decades (see \cite{CLYZ2, HS, Kr,  Lud1,Lud2, LYZ2, LYZ3, Wa, Zh}). For more recent developments, we refer to \cite{Cai, HL1, HL2, HL3}.

%Recent years, a number of new developments have emerged, further advancing the theory (see \cite{Cai, HL1, HL2, HL3})

%%%%经典表面积测度的变分公式、Minkowski问题等等
%%%%介绍近期的Minkowski问题的工作,着重提一下chord的

Variational methods, particularly those based on the first variation of volume or perimeter, have played a fundamental role in the study of geometric inequalities and related extremal problems. In the landmark work \cite{LXYZ2020} by Lutwak-Xi-Yang-Zhang (Xi-LYZ), the authors established a family of variational formulas for chord integrals $I_q(K)$ for $q>0$. In particular, when $q\in(0,1)$, the chord integral $I_q(K)$ is proportional to the fractional perimeter $P_{1-q}(K)$. Motivated from this, we establish the variational formulas for $P_s(K,L)$ and the associated Minkowski problem is solved in Theorem 1.1.

The classical Minkowski problem seeks to characterize convex bodies via their surface area measures, which arise as the first variation of the volume functional under geometric perturbations of the body. Aleksandrov \cite{Ale} established the following variational formula:
\begin{equation}\label{ale var}
    \frac{d}{dt} \Big|_{t=0} V_n([h_0+tf]) = \int_{\sn} f(v) \, dS_{n-1}([h_0], v),
\end{equation}
where 
\[[h_0+tf]=\{x\in\rn:  x\cdot v \leq h_0(v)+f(v),~~\text{for all}~~v\in\sn\},\]
is the Wulff shape generated by $h_0+tf$. Here $h_0, f\in C(\sn)$ are continuous functions on $\sn$, and  $ x\cdot v$ is the canonical inner product in $\rn$. The identity \eqref{ale var} holds whenever $[h_0]$ is a convex body, where $S_{n-1}([h_0],\cdot)$ denotes its surface area measure.

%The classical Minkowski problem seeks to characterize convex bodies via their surface area measures, which arise as the first variation of the volume functional under geometric perturbations of the body. Aleksandrov \cite{Ale} established the following variational formula:
%\[ \frac{d}{dt} \Big|_{t=0} V_n(K_t) = \int_{\sn} f(v) \, dS_{n-1}(K, v),
%\]
%where $S_{n-1}(K,\cdot)$ denotes the surface area measure of $K$, and $f\in C(\sn)$ is a continuous function on the unit sphere $\sn$. The Wulff shape $K_t $  associated with $K$  and $f$ is defined by
%\[
%K_t := \{ x \in \mathbb{R}^n : \langle x, v \rangle \leq h_K(v) + t f(v), \text{ for all } v \in \sn \}.
%\]
%Here, $ \langle \cdot, \cdot \rangle $ denotes the standard inner product in $ \mathbb{R}^n $ and $ h_K $,  the support function of $ K $, is given by $ h_K(v) = \max \{ \langle v, x \rangle : x \in K \} $.
\vskip 6pt

\noindent\textbf{The classical Minkowski problem:} Given a finite Borel measure $ \mu $ on $ \sn $, what are the necessary and sufficient conditions on $ \mu $ such that there exists a convex body $ K \subset \mathbb{R}^n $ satisfying
\[
\mu = S_{n-1}(K, \cdot)?
\]

As a generalization of volume, the notion of mixed volume (to be discussed in detail in Section 2.3) captures more subtle geometric interactions among convex bodies. In particular,  the first mixed volume of $K$ and origin symmetric $L$ can be viewed as the anisotropic perimeter of $K$ with respect to $L$, which is given by
\begin{equation}\label{ani per}
P(K,L)=\int_{\partial K}\|\nu_K(x)\|_{L^*}dx.
\end{equation}
The set $L^*=\{x\in\rn: x\cdot y  \leq 1,~~\text{for all }~~y\in L\}$ is the polar body of $L$ and $\nu_K$ is the unit outer normal of $K$ at $x\in\partial K$. In the special case where $L=B^n$, the anisotropic perimeter $P(K,B^n)$ reduces to the classical perimeter, that is, the surface area of $K$. 
The definition  
$P(K,L)$ naturally extends to bounded Borel sets $E\subset \rn$. In this setting, $\partial K$ is replaced by the reduced boundary of $E$ and $\nu_K$ is replaced by the measure theoretic unit outer normal of $E$. For basic background of sets of finite perimeter, we refer to \cite{Maggi}. By a result of Ludwig \cite{Lud1}, for a bounded Borel set $E$ with finite perimeter, there is
\[\lim_{s\rightarrow 1^-}(1-s)P_s(E,L)=P(E,ZL),\]
where $ZL$ is the moment body of $L$ given by
\[h_{ZL}(v)=\frac{n+1}{2}\int_{L}| v\cdot x |dx.\]

Associated with the mixed volume is the mixed area measure,  which describes the first variation of the mixed volume under Minkowski sum. Given convex bodies $K$ and $Q$ in $\rn$ and $t>0$, we denote by $K+tQ=\{x+ty: x\in K, y\in Q\}$ the Minkowski sum of $K$ and $tQ$. By the linearity of mixed volumes with respect to Minkowski sum and scalar multiplication, one has
\[\frac{d}{dt}\Big|_{t=0^+}P(K+tQ,L)=(n-1)\int_{\sn}h_Q(v)dS_{n-2}(K,L,v).\]
Note that the derivative is one-sided, which poses additional difficulties for the variational method. 
When $L=B^n$, the mixed area measure reduces to the $(n-2)$-th area measure of $K$, denoted by $S_{n-2}(K,\cdot)$, which was introduced by Aleksandrov, Fenchel and Jessen in the 1930s. 

The quermassintegrals, which include volume and surface area as typical cases, form a special family of mixed volumes whose associated Borel measures are denoted by $S_{i}(K,\cdot)$, for $1\leq i\leq n-1$.
%As a special family of mixed volumes, the quermassintegral includes volume and surface area as typical cases. The corresponding Borel measures are denoted by  S_{i}(K,\cdot)$, for $1\leq i\leq n-1$. 
The Minkowski problems concerning area measures $S_{i}(K,\cdot)$ are commonly known as Christoffel-Minkowsi problems.
%The Minkowski problems related to area measures are commonly known as the Christoffel-Minkowski problems. 
The Chirstoffel problem is for the case $i=1$, while the case $i=n-1$ is the classic Minkowski problem. As a longstanding problem, the Christoffel–Minkowski problem has benefited from significant contributions via PDE techniques, resulting in several important partial results (see  \cite{GG, GLL, GM, GMZ}). For the solution of the Christoffel problem, see also \cite{GYY, GZ2, Sch1} and Berg and Firey \cite[Section 8.3.2]{Sch}. We also refer \cite{BHM1,BHM2, CGSF, MU} to some recent results.

An alternative perspective on $S_{n-2}(K,\cdot)$ was proposed by Lutwak-Xi-Yang-Zhang \cite{LXYZ2020} (Xi-LYZ), who developed two-sided variational formulas for chord power integrals 
\[I_q(K)=\int_{{\rm Aff}(n,1)}|K\cap l|^q dl,\]
where ${\rm Aff}(n,1)$ denotes the affine Grassmannian of lines in $\rn$, and $dl$ denotes the Haar measure on it. 
The variational approach leads to a family of 
geometric measures $F_q(K,\cdot)$ for $q>0$, known as chord measures. As $q\rightarrow 0^+$, the chord integral $I_q(K)$ converges to the surface area of $K$. Remarkably, Xi-LYZ \cite{LXYZ2020} shows that $F_q(K,\cdot)$ also  converges to $S_{n-2}(K,\cdot)$, which offers a potential route to the Christoffel–Minkowski problem via suitable limiting arguments.

The landmark work by Xi-LYZ\cite{LXYZ2020} initiated the study of geometric measures in integral geometry and has inspired a variety of subsequent developments, including   $L_p$ and Orlicz chord Minkowski problems; see, for example, \cite{GXZ, HHLW, Li, LCJ, LJ}. More recently, Xi and Zhao \cite{XZ} extended techniques introduced in \cite{LXYZ2020} to develop a new family of affine geometric measures, the limit of which they defined as the affine area measure. Here, “affine” indicates that the measure is either affine covariant or affine contravariant. The associated affine Minkowski problems—Minkowski-type problems for such affine geometric measures—have attracted increasing attention in recent years. For additional results on affine geometric measures, we refer the reader to \cite{BH, BLYZ2013, CLWX, CLWX2, LW}. {We also refer to \cite{HLYZ} for a broader overview of research on Minkowski-type problems.}

%{\cb The $L^p$ chord Minkowski problem was recently studied by }

%%%%引入考虑研究等周不等式的最优常数、极值体的凸性与否所以考虑变分公式等

Our first main result establishes the following variational formula: given an origin symmetric convex body $L$ and $s\in(0,1)$,  
\[\frac{d}{dt}\Big|_{t=0}P_s(K_t,L)=\int_{\sn}f(v)d\tA(K,L,v),\]
where $K_t$ is the Wulff shape generated by $h_K+tf$ with a convex body $K\subset\rn$ and $f\in C(\sn)$. The measure $\tA(K,L,\cdot)$ is called the anisotropic $s$-fractional area measure of $K$ with respect to $L$.

A natural question is to study the limiting behavior of $\tA(K,L,\cdot)$ as $s\rightarrow 0^+$ and $s\rightarrow 1^-$. In section 4, we show that
\[s\tA(K,L,\cdot)\rightarrow S_{n-1}(K,\cdot)\]
as $s\rightarrow 0^+$ and when $K$ has $C^2$ boundary and everywhere positive Gauss curvature,
\[(1-s)\tA(K,L,\cdot)\rightarrow S_{n-2}(K,ZL,\cdot)\]
as $s\rightarrow 1^-$.

In section 5, we study the Minkowski problem of $\tA(K,L,\cdot)$, formulated as Theorem \ref{main}. This problem seeks necessary and sufficient conditions on a Borel measure $\mu$ under for the existence of $K \in \tk^n$ such that
\[\mu = \tA(K,L,\cdot).\] 
In the smooth setting, with $d\mu = gdv$ absolutely continuous, the Minkowski problem can be reformulated as the Monge–Amp{\`e}re equation,
\begin{equation}\label{ma}
    \tilde{V}_{n+s}([h],L,\nabla h(v))\det (\nabla_{\sn}^2h(v)+h(v))=g(v),
\end{equation}
where $\tilde{V}_{n+s}([h],L,\nabla h(v))$ denotes the dual mixed volume of $[h]-\nabla h(v)$ and $L$, as introduced in Section 2.2. The notation $\nabla h$ and $\nabla_{\sn}^2 h$ denotes the Euclidean  gradient and spherical Hessian of $h$ respectively, where $h\in C^2(\sn)$ is continuously twice differentiable and extended to $\rn$ as a $1$-homogeneous function. The smooth Minkowski problem amounts to finding a convex solution of \eqref{ma}, whereas Theorem \ref{main} below resolves the problem without any regularity assumptions.

\begin{theorem}\label{main}
    Let $s\in(0,1)$. Suppose $\mu$ is a finite Borel measure on $\sn$ and $L\subset\rn$ is an origin symmetric convex body. There is a convex body $K\subset\rn$ such that $\mu=\tA(K,L,\cdot)$ if and only if $\mu$ is not concentrated in any subsphere and
    \[\int_{\sn}vd\mu(v)=o.\]
\end{theorem}

In the final section, we establish a necessary condition for the convexity of the optimizer in the anisotropic fractional isoperimetric inequality.

%\begin{theorem}
%    Let $s\in(0,1)$ and $L\in\tk_e^n$. If $K_0$ is a convex body in $\rn$ such that
%\[P_s(K_0,L)=\gamma_s(L)|K_0|^{\frac{n-s}{n}},\]
%then
%\[\tA(K_0,L,\cdot)=cS_{n-1}(K_0,\cdot).\]
%\end{theorem}

\section{Preliminary}
In this section, we collect some basic definitions and results on convex bodies and anisotropic fractional $s$-perimeters. We refer the reader to \cite{Gar,Maggi, Sch} for standard notation and background on convex bodies and sets of finite perimeter.

%%%%%%%如果没什么好分小节的就不分了
\subsection{Convex bodies}
We work in the Euclidean space $\rn$, equipped with the canonical inner product $x\cdot y$ and norm $|x|$ for $x,y\in\rn$. We denote by $\tH^k$ the $k$-dimensional Hausdorff measure in $\rn$, which coincides with the Lebesgue measure on $\rn$. For convenience, we sometimes write $|E|$ to denote the $n$-dimensional Lebesgue measure of a measurable set $E\subset \rn$. Let $B^n$ and $\sn$ denote the unit ball and the unit sphere in $\rn$, respectively. The volume of $B^n$ is denoted by $\omega_n:=|B^n|$.

A set $K \subset \mathbb{R}^n$ is called a convex body if it is compact, convex, and has nonempty interior. 
The family of all compact convex sets in $\rn$ is denoted by $\mathcal{C}^n$, 
and $\mathcal{K}^n \subset \mathcal{C}^n$ denotes the class of convex bodies. 
We write $\mathcal{K}_e^n \subset \mathcal{K}^n$ for the subclass of origin-symmetric convex bodies.

For a compact convex set $F \subset \mathbb{R}^n$, its support function $h_F : \mathbb{S}^{n-1} \to \mathbb{R}$ is given by
\[
h_F(v) = \max\{  v\cdot x  : x \in F \}.
\]
The class $\mathcal{C}^n$ is equipped with the Hausdorff metric $d_H$, which can be expressed in terms of support functions as
\[
d_H(F_1, F_2) = \| h_{F_1} - h_{F_2} \|_\infty = \max_{v \in \sn} |h_{F
_1}(v) - h_{F_2}(v)|,
\]
for any $F_1, F_2 \in\mathcal{C}^n$.

For $K\in\tk^n$ and a boundary point $z\in\partial K$, we say that $v\in\sn$ is a unit outer normal of $K$ at $z$ if $ v\cdot z =h_K(v)$. If $z$ has a unique unit outer normal, then we denote it by $\nu_K(z)$. By \cite[Theorem 2.2.11]{Sch}, the Gauss map $\nu_K(z)$ is defined for almost every $z\in\partial K$. For a Borel set $\eta\subset\sn$, the inverse Gauss map $\nu_K^{-1}(\eta)$ is defined by
\[\nu_K^{-1}(\eta)=\{z\in\partial K: \exists v\in \eta~~\text{such that}~~ v\cdot z=h_K(v)\}.\]
For $z\in \partial K$ and $v\in \nu_K^{-1}(\{z\})$, the hyperplane $H=\{x\in\rn: x\cdot v=h_K(v)\}$ is called the support hyperplane of $K$, which is also the tangent space of $\partial K$ at $z$ and denoted by $T_zK$.

The surface area measure $S_{n-1}(K,\cdot)$ is then given by
\[S_{n-1}(K,\eta)=\tH^{n-1}(\nu_K^{-1}(\eta)),\]
for each Borel subset $\eta\subset\sn$.

\subsection{Mixed volumes}

The mixed volume of  $K_1,\ldots ,K_n\in\mathcal{C}^n$, denoted by $V(K_1,\ldots, K_n)$, is nonnegative and symmetric, that is, 
\[V(K_1,\ldots,K_n)=V(K_{\sigma(1)},\ldots, K_{\sigma(n)}),\]
for every permutation $\sigma$ of $\{1,\ldots, n\}$.

The functional $V: (\mathcal{C}^n)^n\rightarrow [0,\infty)$ is also multilinear with respect to Minkowski addition and nonnegative scalar multiplication:
\[V(aK+bL,K_2,\ldots, K_n)=aV(K,K_2,\ldots,K_n)+b(L,K_2,\ldots, K_n)\]
for $a,b\ge 0$. 

Recall that every twice continuously differentiable function $f:\sn\rightarrow \tr$, i.e. $f\in C^2(\sn)$, can be written as the difference of two support functions and that $C^2(\sn)$ is dense in $C(\sn)$. By the Riesz representation theorem, there exists a Borel measure on $\sn$, denoted by $S(K_2,\ldots, K_n,\cdot)$, such that
\[V(K_1,\ldots, K_n)=\int_{\sn}h_{K_1}(v)dS(K_2,\ldots, K_n,v).\]
This measure $S(K_2,\ldots, K_n,\cdot)$ is called the mixed area measure.

An important special case arises when $K=K_1=\ldots =K_{i}$ and $B^n=K_{i+1}=\ldots =K_n$. In this case, the mixed volume is denoted by 
\[W_{n-i}(K)=V(\underbrace{K,\ldots, K}_{i}, \underbrace{B^n,\ldots, B^n}_{n-i}),\]
which is called the quermassintegral for $i\in\{0,1,\ldots, n\}$. Here, $i=n$ corresponds to the volume of $K$, while $i=n-1$ corresponds to its surface area. The mixed area measure associated with $W_i(K)$ is denoted by $S_{i-1}(K,\cdot)$ and is called the $(i-1)$-st area measure of $K$. In particular, $S_{n-1}(K,\cdot)$ is the surface area measure.

\subsection{Dual mixed volumes} Let $K\in\tk^n$ and $z\in K$. The radial function of $K$ with respect to $z$ is defined by

\[\rho_{K,z}(u)=\max\{t\ge0: z+tu\in K\},~~u\in\sn.\]
When $z=o$, we simply write $\rho_K:=\rho_{K,o}$. 

For $q\in\tr$ and $L\in \tk_e^n$, the $q$-th dual mixed volume of $K$ and $L$ with respect to $z$ is defined by
 
\[\tilde{V}_{n-q}(K,L,z)=\frac{1}{n}\int_{\mathbb S_z^+}\rho_{K,z}(u)^{q}\rho_L(u)^{n-q}du,\]
where
\[\mathbb S_z^+=\{u\in\sn: \rho_{K,z}(u)>0\}.\]
For $z \in \operatorname{int} K$, the set $\mathbb S_z^+$ coincides with the entire sphere $\sn$,  
whereas for $\tH^{n-1}$-almost every $z \in \partial K$, it is either an open or a closed hemisphere.  
We denote by $\tilde{V}_{n-q}(K, z)$ the dual mixed volume with respect to $L = B^n$.

Note that when $z\in$ int $K$ and $q\in\tr$, the quantity $\tilde{V}_{n-q}(K,L,z)$ is always finite. For $z\in\partial K$ and $q>0$, since $\rho_L$ is bounded, the integrand is bounded and hence $\tilde{V}_{n-q}(K,L,z)$ remains finite. When $z\in\partial K$ and $q<0$, the situation becomes more delicate due to the potential singularity of $\rho_{K,z}(u)$. 

By a result of Xi–LYZ \cite[Lemma 4.9]{LXYZ2020}, the function \( \tilde{V}_{n-q}(K, \cdot) \) is integrable over \( \partial K \) for every \( q > -1 \). As \( L \in \mathcal{K}_e^n \) is a bounded convex body, its radial function satisfies \( \rho_L(u) \leq c \) for some constant \( c > 0 \) and all \( u \in \sn \). Consequently,
\[
\tilde{V}_{n-q}(K,L,z) \leq c^{n-q} \tilde{V}_{n-q}(K,z)
\]
for all \( z \in \partial K \).  
It then follows that \( \tilde{V}_{n-q}(K,L,\cdot) \) is integrable over \( \partial K \) for every \( q \in (-1,0) \).  
In particular, for every such \( q \), \( \tilde{V}_{n-q}(K,L,z) < \infty \) for almost every \( z \in \partial K \).

\subsection{Anisotropic fractional $s$-perimeter}
Suppose $E \subset \mathbb{R}^n$ is a bounded Borel set and $L \in \mathcal{K}_e^n$. Then:

\begin{itemize}[itemsep=5pt, label=\scriptsize$\bullet$]
    \item $P_s(E,L)=P_s(E^c,L)$.
    
    \item $P_s(E+x,L)=P_s(E,L)$, for any $x\in\rn$.
    
    \item $P_s(\lambda E,L)=\lambda^{n-s}P_s(E,L)$, for any $\lambda>0$, where $\lambda E=\{\lambda x: x\in E\}$.
    
    \item $P_s(E,L)$ is lower semicontinuous with respect to $E$. That is, if $|E\triangle E_i|\rightarrow 0$ as $i\rightarrow \infty$, then $P_s(E,L)\leq \liminf_{i\rightarrow\infty} P_s(E_i,L)$.

\end{itemize}

In the following, we focus on the case $E = K$, where $K$ is a convex body. Setting $z=x-y$ and applying polar coordinates gives
\[P_s(K,L)=\frac{1}{s}\int_{\sn}\rho_L(u)^{n+s}\int_K\rho_{K,y}(u)^{-s}dydu.\]
We may write $y=y^{\prime}+tu$ with $y^{\prime}\in u ^{\perp}$, so that
\[\int_K\rho_{K,y}(u)^{s}dy=\int_{K|u^{\perp}}\int_{f_K(y^{\prime})}^{g_K(y^{\prime})}(g_K(y^{\prime})-t)^{-s}dtd\tH^{n-1}(y^{\prime}),\]
where $f_K(y^{\prime})=\min\{t\in\tr: y^{\prime}+tu\in K\}$ and $g_K(u)=\max\{t\in\tr: y^{\prime}+tu\in K\}$. Hence, 
\[P_s(K,L)=\frac{1}{s(1-s)}\int_{\sn}\rho_L(u)^{n+s}\int_{K|u^{\perp}}(g_K(y^{\prime})-f_K(y^{\prime}))^{1-s}d\tH^{n-1}(y^{\prime})du.\]

The quantity $g_K(y^{\prime})-f_K(y^{\prime})$ is denoted by $X_K(y^{\prime},u)$, and is referred to as the X-ray function of $K$. Equivalently, it can be defined on $\rn\times \sn$ by
\[X_K(y,u)=\tH^1(K\cap (y+\tr u)).\]
This function records the length of the intersection between $K$ and the line $y + \mathbb{R} u$. In terms of $X_K$, the anisotropic fractional perimeter can be expressed in terms of $X_K$ as
\begin{equation} \label{as-X}
P_s(K, L) = \frac{1}{s(1-s)} \int_{\sn} \rho_L(u)^{n+s} \int_{u^\perp} X_K(y, u)^{1-s} \, d\mathcal{H}^{n-1}(y) \, du.
\end{equation}

In particular, when $L = B^n$, the functional $P_s(K) := P_s(K, B^n)$ is proportional to the chord power integral of $K$, namely,
\begin{equation}\label{Ps-int}
    P_s(K) = \frac{n\omega_n}{s(1-s)} I_{1-s}(K) = \frac{n\omega_n}{s(1-s)} \int_{\mathrm{Aff}(n,1)} |K \cap l|^{1-s} \, dl,
\end{equation}
which follows from the Blaschke-Petkantschin Formula (\cite[Theorem 7.2.7]{SW}).

%For a bounded Borel set $E\subset\rn$ and $L\in\tk_e^n$, it was shown in \cite{Lud1} that
%\[\lim_{s\rightarrow 0^+}sP_s(E,L)=n|E||L|,\]
%and
%\[\lim_{s\rightarrow 1^-}(1-s)P_s(E,L)=P(E, ZL).\]

\section{The variational formulas and anisotropic fractional area measures}

This section focuses on the variational formula for the anisotropic fractional perimeter and the associated geometric measures. Before proving the variational formula, we introduce the anisotropic fractional area measure, which naturally arises in the process.

\begin{definition}
    Let $s\in(0,1)$, $K\in\tk^n$ and $L\in\tk_e^n$. The anisotropic $s$-fractional area measure of $K$ with respect to $L$ is given by
\[\mathcal{A}_s(K,L,\eta)=\frac{n}{s}\int_{\nu_K^{-1}(\eta)}\tilde{V}_{n+s}(K,L,z)d\tH^{n-1}(z),\]
where $\eta\subset \sn$ is a Borel set.
\end{definition}

Let $K$ be a convex body in $\rn$ and $f:\sn\rightarrow \tr$ be a continuous function on $\sn$. We write
\begin{equation}\label{Wulff Shape}
    K_t=\{x\in\rn:  x\cdot v \leq h_K(v)+tf(v),~~\text{for all}~~v\in\sn\},
\end{equation}
which is a convex body in $\rn$, for sufficiently small $\delta>0$ and $|t|<\delta$.

We will use the following result in \cite{LXYZ2020}, which provides the derivative of the X-ray function.

\begin{lemma}\cite[Lemma 5.2]{LXYZ2020}\label{XLYZ} Suppose $K\in\tk^n$, $f:\sn\rightarrow \tr$ is a continuous function and $K_t$ is the Wulff shape. If $u\in\sn$, then for almost all $y$ in the interior of $K|u^{\perp}$,
\[\frac{dX_{K_t}(y,u)}{dt}\Big|_{t=0}=\frac{f(\nu_K(y^+))}{u\cdot\nu_K(y^+)}-\frac{f(\nu_K(y^-))}{u\cdot \nu_K(y^-)},\]
where $y^+$ and $y^-$ are the upper and lower points of $\partial K\cap (y+\tr u)$.
\end{lemma}

For the variational formula, we need the following identity for the anisotropic $s$-fractional area measure. A similar formula in the case $L = B^n$ appears in \cite{LXYZ2020}, and the argument carries over to the general case without change.

\begin{lemma}\label{id}
    Let $s\in(0,1)$, $L\in\tk_e^n$ and $K\in\tk^n$. Suppose $f$ is a continuous function on $\sn$. Then
    \[2n\int_{\partial K}f(\nu_K(z))\tilde{V}_{n+s}(K,L,z)d\tH^{n-1}(z)=\int_{\sn}\rho_L(u)^{n+s}\int_{\partial K}X_K(z,u)^{-s}f(\nu_K(z))d\tH^{n-1}(z)du.\]
\end{lemma}

\begin{proof}
    By the definition of $\tilde{V}_{n+s}(K,L,z)$,
    \[\int_{\partial K}f(\nu_K(z))\tilde{V}_{n+s}(K,L,z)d\tH^{n-1}(z)=\frac{1}{n}\int_{\partial K}f(\nu_K(z))\int_{\mathbb S_z^+}\rho_L(u)^{n+s}\rho_{K,z}(u)^{-s}dud\tH^{n-1}(z).\]
    
    For $u\in\sn$, we denote by
    \[\partial K(u)=\{z\in\partial K: \rho_{K,z}(u)>0\}.\]
    Then by Fubini's theorem and the fact that $\rho_L(u)=\rho_L(-u)$, we have
    \begin{equation}\label{1}
        \begin{aligned}
        \int_{\partial K}f(\nu_K(z))\tilde{V}_{n+s}(K,L,z)&d\tH^{n-1}(z)\\
        &=\frac{1}{n}\int_{\sn}\rho_L(u)^{n+s}\int_{\partial K(u)}f(\nu_K(z))\rho_{K,z}(u)^{-s}d\tH^{n-1}(z)du\\
        &=\frac{1}{n}\int_{\sn}\rho_L(u)^{n+s}\int_{\partial K(-u)}f(\nu_K(z))\rho_{K,z}(u)^{-s}d\tH^{n-1}(z)du.
        \end{aligned}
    \end{equation}

    Note that for almost every $u\in\sn$, the intersection $\partial K\cap (z+\tr u)$ consists of no more than two points. Combining with the fact that $\rho_{K,z}(u)=X_K(z,u)$ and \eqref{1}, we obtain
    \[\int_{\partial K}f(\nu_K(z))\tilde{V}_{n+s}(K,L,z)d\tH^{n-1}(z)=\frac{1}{2n}\int_{\sn}\rho_L(u)^{n+s}\int_{\partial K}X_K(z,u)^{n+s}f(\nu_K(z))d\tH^{n-1}(z)du,\]
    which concludes the proof.
\end{proof}

\begin{theorem}\label{var}
Suppose $s \in (0,1)$, $L \in \tk_e^n$, $K \in \tk^n$ and $f$ is a continuous function on $\sn$. Let $K_t$ be the Wulff shape  as defined  in \eqref{Wulff Shape}. Then 
    \[\frac{d}{dt}\Big|_{t=0}P_s(K_t,L)=2n\int_{\sn}f(u)d\mathcal{A}_s(K,L,u).\]
\end{theorem}

\begin{proof}
{ By the result in \cite[Lemma 5.4]{LXYZ2020}}, there is a class of nonnegative integrable functions $\phi_t(y,u)$ defined for $u\in\sn$ and $y\in u^{\perp}$ such that
\[\Big|\frac{1}{t}\Big(X_{K_t}(y,u)^{1-s}-X_K(y,u)^{1-s}\Big)\Big|\leq \phi_t(y,u),\]
and
\[\lim_{t\rightarrow 0}\int_{\sn}\int_{u^{\perp}}\phi_t(y,u)d\tH^{n-1}(y)du=\int_{\sn}\int_{u^{\perp}}\lim_{t\rightarrow 0}\phi_t(y,u)d\tH^{n-1}(y)du,\]
where $\lim_{t\rightarrow 0}\phi_t(y,u)$ is integrable.

    Since $L\in \tk_e^n$ is bounded, there exists $r>0$ such that $\rho_L(u)<r$ for all $u\in\sn$. Hence, 
    \[\Big|\frac{1}{t}\rho_L(u)^{n+s}\Big(X_{K_t}(y,u)^{1-s}-X_K(y,u)^{1-s}\Big)\Big|\leq r^{n+s}\phi_t(y,u).\]
    By the representation \eqref{as-X}, Lemma \ref{XLYZ} and the dominated convergence theorem, it follows that
    \begin{align*}
        \frac{d}{dt}\Big|_{t=0}P_s(K_t,L)
        &=\lim_{t\rightarrow 0}\frac{1}{s(1-s)}\int_{\sn}\int_{u^{\perp}}\rho_L(u)^{n+s} \frac{X_{K_t}(y,u)^{1-s}-X_K(y,u)^{1-s}}{t}d\tH^{n-1}(y)du\\
        &=\frac{1}{s}\int_{\sn}\int_{u^{\perp}}\rho_L(u)^{n+s} X_K(y,u)^{-s}\Big(\frac{f(\nu_K(y^+))}{u\cdot\nu_K(y^+)}-\frac{f(\nu_K(y^-))}{u\cdot \nu_K(y^-)}\Big)d\tH^{n-1}(y)du.
    \end{align*}
     Note that $d\tH^{n-1}_{u^{\perp}}(y)=|z\cdot u|d\tH^{n-1}_{\partial K}(z)$, where $y=z|_{u^{\perp}}$. Therefore, we have
    \[\frac{d}{dt}\Big|_{t=0}P_s(K_t,L)=\frac{1}{s}\int_{\sn}\rho_L(u)^{n+s}\int_{\partial K}X_K(z,u)^{-s}f(\nu_K(z))d\tH^{n-1}(z)du.\]
    Combining with Lemma \ref{id}, we have
    \begin{align*}
    \lim_{t\rightarrow 0}\frac{P_s(K_t,L)-P_s(K,L)}{t}&=\frac{2}{s}\int_{\partial K}f(\nu_{K}(z))\int_{\mathbb S_z^+}\rho_{K,z}(u)^{-s}\rho_L(u)^{n+s}dud\tH^{n-1}(z)\\
    &=2n\int_{\sn}f(v)d\mathcal{A}_s(K,L,v),
    \end{align*}
    which concludes the proof.
\end{proof}

We now turn to some basic properties of the measure $\tA (K,L,\cdot)$. Note that $P_s(K,B^n)$ is proportional to the chord integral, and that $P_s(K,L)$ is invariant under translations of $K$, i.e.,  $P_s(K,L)=P_s(K+x_0,L)$ for any $x_0\in\rn$. Moreover, it is positively homogeneous of degree $n-s$ in $K$, that is, $P_s(\lambda K,L)=\lambda^{n-s}P_s(K,L)$ for all $\lambda>0$. The measure $\tA(K,L,\cdot)$ satisfies the following properties:

\begin{itemize}[itemsep=5pt, label=\scriptsize$\bullet$]
    \item When $L=B^n$, the measure $\tA(K,B^n,\cdot)$ is proportional to the $(1-s)$-th chord measure.
    \item $\tA(K,L,\cdot)=\tA(K+x_0,L,\cdot)$ for any $x_0\in\rn$.
    %; that is, $\tA(K,L,\cdot)$ is  invariant under translations of the first argument.
    \item $\tA(\lambda K,L,\cdot)=\lambda^{n-s}\tA(K,L,\cdot)$ for any $\lambda>0$.
\end{itemize}

 In particular, we show that its centroid lies at the origin. To this end, we first recall the following technical lemma from \cite{XZ}.

%这个是不是自己写个证明比较好,或者把证明融进Lem 3.6的证明
\begin{lemma}\cite[Lemma 3.2]{XZ}\label{xz-lem1}
    Let $q>-1$ and $K\in\tk^n$. If $F:K\rightarrow [0,\infty)$ is concave, and $F>0$ in the interior of $K$ and vanishes on $\partial K$, then
    \[\int_K  x\cdot  \nabla (F^{q+1}(x)) dx=-n\int_{K} F(x)^{q+1} dx.\]
\end{lemma}

\begin{lemma}\label{centroid}
    Suppose $L\in\tk_e^n$ and $K\in\tk^n$. Then
    \[\int_{\sn}v d\mathcal{A}_s(K,L,v)=o\]
\end{lemma}

\begin{proof}
%    First we aim to show
%    \begin{equation}\label{int 1}
%        \begin{aligned}
%        P_s(K,L)=\int_{\sn} h_K(v)d \mathcal{A}_s(K,L,v).
%    \end{aligned}
%    \end{equation}

   By the definition of $\mathcal{A}_s(K,L,\cdot)$ and Lemma \ref{id}, we have
   \begin{align*}
   \int_{\sn}h_K(v)d\mathcal{A}_s(K,L,v)&=\frac{n}{s}\int_{\partial K} z\cdot \nu_K(z) \tilde{V}_{n+s}(K,L,z)d\tH^{n-1}(z)\\
   &=\frac{1}{2s}\int_{\sn}\rho_L(u)^{n+s}\int_{\partial K} z\cdot\nu_K(z) X_K(z,u)^{-s}d\tH^{n-1}(z)du.
   \end{align*}
   For fixed $u\in\sn$, one can write
   \[K=\{y+tu: y\in K|u^{\perp}, f_K(y)\leq t\leq g_K(y) \},\]
   where $f_K(y)$ and $g_K(y)$ are given in Section 2.4.
   
   Since $K$ is convex, $f_K$ and $-g_K$ are convex on $K|u^{\perp}$ and therefore locally Lipschitz. Then $X_K(y,u)=g_K(y)-f_K(y)$ is concave on $K|u^{\perp}$. Note that for almost all $u\in\sn$, the boundary $\partial K$ does not contain any line parallel to $u$. With such a direction $u$ fixed, it follows that for almost all $y\in K|u^{\perp}$, 
   \[\nu_K(y+f_K(u)u)=\frac{\nabla f_K(y)-u}{\sqrt{1+|\nabla f_K(y)|^2}},~~\nu_K(y+g_K(y)u)=\frac{-\nabla g_K(y)+u}{\sqrt{1+|\nabla g_K(y)|^2}}.\]  
   We then have
   \begin{align*}
       \int_{\partial K} z\cdot\nu_K(z) X_K(z,u)^{-s}d\tH^{n-1}(z)=&\int_{K|u^{\perp}}( y\cdot\nabla f_K(y) -f_K(y))X_K(y,u)^{-s}d\tH^{n-1}(y)\\
       +&\int_{K|u^{\perp}}(- y\cdot\nabla g_K(y) +g_K(y))X_K(y,u)^{-s}d\tH^{n-1}(y).
   \end{align*}

   Since $g_K(y)-f_K(y)=X_K(y,u)$, we obtain
   \begin{align*}
       \int_{\partial K} z\cdot \nu_K(z) X_K(z,u)^{-s}d\tH^{n-1}(z)=&\int_{K|u^{\perp}}X_K(y,u)^{1-s}d\tH^{n-1}(y)\\
       &-\int_{K|u^{\perp}} y\cdot\nabla X_K(y,u) X_K(y,u)^{-s}d\tH^{n-1}(y).
   \end{align*}
   It then follows from Lemma \ref{xz-lem1}, with $q=-s$, that
   \begin{align*}
       \int_{\partial K} z\cdot \nu_K(z) &X_K(z,u)^{-s}d\tH^{n-1}(z)\\&=\int_{K|u^{\perp}}\Big(X_K(y,u)^{1-s}-\frac{1}{1-s} y\cdot \nabla (X_K(y,u)^{1-s}) \Big)d\tH^{n-1}(y)\\
       &=\frac{n-s}{1-s}\int_{K|u^{\perp}} X_K(y,u)^{1-s}d\tH^{n-1}(y).
   \end{align*}
   Consequently,
   \[\int_{\sn}h_K(v)d\tA(K,L,v)=\frac{n-s}{2s(1-s)}\int_{\sn}\rho_L(u)^{n+s}\int_{K|u^{\perp}}X_K(y,u)^{1-s}d\tH^{n-1}du.\]
   By \eqref{as-X}, we finally have
   \begin{equation}\label{as-int}
      P_s(K,L)=\frac{2}{n-s}\int_{\sn}h_K(v)d\tA(K,L,v). 
   \end{equation}
   
   Since $P_s(K+x_0,L)=P_s(K,L)$ for any $x_0\in\rn$, by \eqref{as-int} and the translation invariance of $\tA(K,L,\cdot)$, we obtain
   \begin{align*}
   \int_{\sn}h_K(v) d\tA(K,L,v)&=\int_{\sn}(h_K(v)+ x_0\cdot v)d\tA(K+x_0,L,v)\\&=\int_{\sn}(h_K(v)+ x_0\cdot v)d\tA(K,L,v),
   \end{align*}
   which implies that
   \[x_0\cdot\int_{\sn} v  d\tA(K,L,v)=0.\]
   By the arbitrariness of $x_0$, we conclude the proof.
\end{proof}

%Since $P_s(K+x_0,L)=P_s(K,L)$ for any $x_0\in\rn$, Theorem \ref{var} implies that the anisotropic $s$-fractional area measure is translation invariant, which is the following corollary.

%\begin{corollary}\label{trans-inv}
%    Suppose $K\in\tk^n$ and $L\in\tk_e^n$. Then
%    \[\tA(K+x_0,L,\cdot)=\tA(K,L,\cdot),\]
%    for any $x_0\in\rn$.
%\end{corollary}

\section{Limiting behavior of anisotropic fractional area measures}

The limiting behavior of the anisotropic fractional $s$-perimeter has been studied by Ludwig \cite{Lud1}, revealing an interesting connection to classical geometric invariants. Specifically, it was shown that

\[\lim_{s\rightarrow 0^+}sP_s(E,L)=n|E||L|,\]
and
\[\lim_{s\rightarrow 1^-}(1-s)P_s(E,L)=P(E,ZL),\]
where $E\subset\rn$ is a bounded Borel set with finite perimeter and $ZL$ is the moment body of $L$.

Motivated by these results, we study the corresponding limiting behavior of anisotropic $s$-fractional area measures, which arise as the geometric derivatives of $P_s(\cdot, L)$ in the setting of convex bodies.

The following theorem shows that  the anisotropic $s$-fractional area measure converges to the surface area measure as $s\rightarrow 0^+$

\begin{theorem}
    If $K\in\tk^n$ and $L\in\tk_e^n$, then
    \[s\tA(K,L,\cdot)\rightarrow \frac{|L|}{2}S_{n-1}(K,\cdot),\]
    as $s\rightarrow 0^+$
\end{theorem}

\begin{proof}
    Suppose $(s_i)$ is an arbitrary convergent sequence in $(0,1)$ and $s_i\rightarrow 0$ as $i\rightarrow \infty$. 
    Let $f\in C(\sn)$. Then $f$ is bounded on $\sn$; that is, there exists $c>0$ such that $|f(v)|\leq c$ for all $v\in\sn$. By the definition of $\tA(K,L,\cdot)$,
    \[s_i\int_{\sn}f(v)d{\mathcal A}_{s_i}(K,L,v)=n\int_{\partial K} f(\nu_K(z))\tilde{V}_{-s_i}(K,L,z)d\tH^{n-1}(z).\]
    
    Since $L\in\tk_e^n$, there exists $l>0$ such that
    $\rho_L(u)^{n+s}\leq l$ 
    for all $u\in\sn$ and $s\in(0,1)$. For an arbitrary $\epsilon>0$, note that $\rho_{K,z}(u)+\epsilon$ is bounded on $\sn$. Then there is $M>0$ such that
    \[\tilde{V}_{-s_i}(K,L,z)\leq l\tilde{V}_{-s_i}(K,z)\leq \frac{l}{n}\int_{\sn}(\rho_{K,z}(u)+\epsilon)^{-s_i}du \leq \omega_n lM^{-s_i}, \]
    and hence 
    \[f(\nu_K(z))\tilde{V}_{-s_i}(K,L,z)\leq \omega_n clM^{-s_i}.\]

    Note that
    \[\lim_{i\rightarrow\infty}\omega_n cl\int_{\partial K}M^{-s_i}d\tH^{n-1}(z)=\omega_n cl\tH^{n-1}(\partial K)=\omega_n cl\int_{\partial K}\lim_{i\rightarrow \infty}M^{-s_i}d\tH^{n-1}(z).\]
    Therefore by the dominated convergence theorem, we have
    \[\lim_{i\rightarrow\infty}\int_{\partial K} f(\nu_K(z))\tilde{V}_{-s_i}(K,L,z)d\tH^{n-1}(z)=\int_{\partial K}f(\nu_K(z))\lim_{i\rightarrow \infty} \tilde{V}_{-s_i}(K,L,z)d\tH^{n-1}(z).\]

    Recall that $\tilde{V}_{n+s}(K,L,z)$ is finite for almost every $z\in\partial K$. After removing a set of measure zero $\eta$, we may assume that $\tilde{V}_{-s_i}(K,L,z)$ and $\tilde{V}_{-s_i}(K,z)$ are finite for all $z\in\partial K\backslash \eta$ and for each $i\in\mathbb N$. For $z\in\partial K\backslash\eta$, we have that $\rho_{K,z}^{-s_i}$ is integrable on $\mathbb S_z^+$. Moreover, 
    \[\rho_{K,z}(u)^{-s_i}\rho_L(u)^{n+s_i}\leq l\rho_{K,z}(u)^{-s_i}.\]
    The dominated convergence theorem now implies
    \[\lim_{i\rightarrow \infty}\tilde{V}_{-s_i}(K,L,z)=\lim_{i\rightarrow\infty}\frac{1}{n}\int_{\mathbb S_z^+}\rho_{K,z}(u)^{-s_i}\rho_L(u)^{n+s_i}du=\frac{1}{2n}\int_{\sn}\rho_{L}(u)^{n}du=\frac{|L|}{2}.\]
    Hence, we have
    \[\lim_{i\rightarrow\infty}s_i\int_{\sn}f(v)d{\mathcal A}_{s_i}(K,L,v)=\frac{|L|}{2}\int_{\partial K}f(\nu_K(z))d\tH^{n-1}(z)=\frac{|L|}{2}\int_{\sn}f(v)dS_{n-1}(K,v),\]
    which concludes the proof.
\end{proof}

{As $s\rightarrow 1^-$, the limiting behavior of $\tA(K,L,\cdot)$ becomes more subtle. We consider the smooth case here.} In what follows, we focus on the case where $K\in\tk^n$ has $C^2$ boundary with everywhere positive Gauss curvature. That is, at each point $z\in\partial K$, the principal curvatures $\kappa_1,\ldots, \kappa_{n-1}$ are all positive. We denote by $\kappa_K(z)$ the Gauss curvature and by $\kappa_K(z,u)$ the normal curvature in direction $u\in T_z K$.

Let $h_K$ be the support function of $K$, and let  $\nabla h_K(v)\in\rn$ and $\nabla^2h_K(v):\rn\rightarrow \rn$ 
denote its Euclidean gradient and Hessian, respectively. For $v\in\sn$, the point $z=\nabla h_K(v)$ lies on $ \partial K$. The principal radii of curvature which are the  eigenvalues of $\nabla^2h_K(v)$ at $z$, are given by $1/\kappa_1,\ldots , 1/\kappa_{n-1}$. The associated eigenvectors $e_1,\ldots, e_{n-1}$ are the principal directions. The normal curvature of $K$ at $z$ in the direction $u$ is given by
\[\kappa(z,u)=\kappa_1 (u\cdot e_1)^2+\cdots +\kappa_{n-1}( u\cdot e_{n-1})^2.\]

\begin{lemma}\label{conv} Suppose $K\in\tk^n$ has $C^2$ boundary and everywhere positive Gauss curvature. Let $L\in\tk_e^n$ and $z\in\partial K$. Then
    \[\lim_{s\rightarrow 1^-}(1-s)\int_{\mathbb S_z^+}\rho_L(u)^{n+ s}X_K(z,u)^{-s}du=\int_{\sn\cap \nu_K(z)^{\perp}}\rho_L(\theta)^{n+1}\kappa_K(z,\theta)d\tH^{n-2}( \theta).\]
\end{lemma}

\begin{proof}
    Since $X_K(z,u)$ and $\kappa_K(z,\theta)$ are invariant under the translation of $z$, we assume $z=o$ in what follows.    
    Let $e_n=-\nu_K(o)$ and let $e_1,\ldots ,e_{n-1}$ be the eigenvectors of $\nabla^2h_K(-e_n)$, with  corresponding eigenvalues $1/\kappa_1,\ldots, 1/\kappa_{n-1}$. The basis $e_1,\ldots, e_n$ is chosen to be orthonormal, since $\nabla^2h_K(-e_n)$ is self-adjoint.
    
    In a small neighborhood of $o$, the boundary $\partial K$ can then be locally represented as the graph of a function:
    \begin{equation}\label{loc}
        y_n=f(y^{\prime})=\frac{1}{2}(\kappa_1y_1^2+\cdots +\kappa_{n-1}y_{n-1}^2)+o(|y^\prime|^2),
    \end{equation}
    where $y=(y^{\prime},y_n)=(y_1,\ldots,y_{n-1},y_n)$ is the coordinate with respect to the orthonormal frame $(o; e_1,\ldots,e_n)$.

    In this coordinate system, $X_K(o,u)=\rho_K(u)$ is positive if and only if $ u\cdot e_n>0$. We denote the corresponding integration domain by $D=\{u\in\sn:  u\cdot e_n>0\}$. 
    
    Let $\epsilon>0$ be arbitrary. Suppose $y^{\prime}=r\theta$, where $\theta\in \mathbb S^{n-2}$. Note that  $f(r\theta)=O(r^2)$, when $r>0$ is sufficiently small. It follows that
    %\begin{align*}
    %    \left|\frac{(r\theta, f(r\theta))}{\sqrt{r^2+f(r\theta)^2}}-(\theta,0)\right|^2=2-\frac{2r}{\sqrt{r^2+f(r\theta)^2}}=2-\frac{2}{\sqrt{1+f(r\theta)^2/r^2}}\rightarrow 0,
    %\end{align*}

    \[
    \left|\frac{(r\theta, f(r\theta))}{\sqrt{r^2+f(r\theta)^2}} - (\theta, 0)\right|^2
    = 2 - \frac{2}{\sqrt{1 + f(r\theta)^2 / r^2}} \to 0,
    \]
     as $r\rightarrow 0^+$. Therefore we choose $\delta>0$ such that \eqref{loc} holds for $|y^{\prime}|<\delta$ and 
    \begin{equation}\label{lem 1 appro 1}
        \left|\rho_L\bigg(\frac{(r\theta, f(r\theta))}{\sqrt{r^2+f(r\theta)^2}}\bigg)^{n+1}-\rho_L(\theta,0)^{n+1}\right|<\epsilon,
    \end{equation}
    when $r\in(0,\delta)$. We denote by
    \[D_1=\{u\in D: |y^{\prime}|\in (0,\delta), \text{where}~~(y^{\prime},f(y^{\prime}))=\rho_K(u)u\in\partial K \}.\]
    
    %Let $d\tH^{n-1}(y)$ be the surface area element of $\partial K$ and let $d\tH^{n-1}(y^{\prime})$ be the volume element in $e_n^{\perp}$. We have

    Recall that
    \begin{equation}\label{coor trans1}
        \nu_K(y)=\frac{(\nabla f(y^{\prime}),-1)}{\sqrt{1+|\nabla f(y^{\prime})|^2}},\quad d\tH^{n-1}_{\partial K}(y)=\sqrt{1+|\nabla f(y^{\prime})|^2}d\tH^{n-1}_{e_n^{\perp}}(y^{\prime}).
    \end{equation}
    For the spherical Lebesgue measure $du$ and $y^{\prime}\in e_n^{\perp}$ with $|y^{\prime}|\in(0,\delta)$, we also have
    \begin{equation*}
        \rho_K(u)^ndu=| y\cdot \nu_K(y) |d\tH^{n-1}(y),\quad y^{\prime}\cdot  \nabla f(y^{\prime})=2f(y^{\prime})+o(|y^{\prime}|^2).
    \end{equation*}
    Combining with \eqref{coor trans1}, we obtain
    \begin{equation}\label{coor trans2}
        \rho_K(u)^ndu=(f(y^{\prime})+o(|y^{\prime}|^2))d\tH^{n-1}(y^{\prime})
    \end{equation}

    We first consider the integral on $D_1$. By \eqref{coor trans2} and the coordinate transform $\rho_K(u)u=(y^{\prime}, f(y^{\prime}))$ , we have
    \begin{align*}
        (1-&s)\int_{D_1}\rho_L(u)^{n+s}\rho_K(u)^{-s}du\\
        &=(1-s)\int_{\{y^{\prime}\in e_n^{\perp}:~|y^{\prime}|<\delta\}}\rho_L\Big(\frac{(y^{\prime},f(y^{\prime}))}{|(y^{\prime},f(y^{\prime}))|}\Big)^{n+s}|(y^{\prime},f(y^{\prime}))|^{-s-n}(f(y^{\prime})+o(|y^{\prime}|^2))d\tH^{n-1}(y^{\prime})\\
        &=(1-s)\int_{\mathbb S^{n-2}}\int_0^{\delta}\rho_L\Big(\frac{(r\theta,f(r\theta))}{|(r\theta,f(r\theta))|}\Big)^{n+s}|(r\theta,f(r\theta))|^{-s-n}(f(r\theta)+o(r^2))r^{n-2}drd\tH^{n-2}(\theta).
    \end{align*}
    Note that $f(r\theta)=\frac{1}{2}r^2\kappa_K(o,\theta)+o(r^2)$. Then
    \begin{align*}
        &(1-s)\int_{D_1}\rho_L(u)^{n+s}\rho_K(u)^{-s}du\\
        \leq &(1-s)\int_{\mathbb S^{n-2}}\int_0^{\delta} \rho_L\Big(\frac{(r\theta,f(r\theta))}{|(r\theta,f(r\theta))|}\Big)^{n+s}\Big(1+\frac{r^2\kappa_K(o,\theta)^2}{4}+o(r^2)\Big)^{\frac{-s-n}{2}}r^{-s}\\ 
        & \quad\quad\quad\quad\quad\quad\quad\quad\quad\quad\quad\quad\quad\quad\quad\quad\quad\quad\quad\quad\quad\quad\quad\quad\quad\quad\cdot\Big(\frac{\kappa_K(o,\theta)}{2}+o(1)\Big)drd\tH^{n-2}(\theta)\\
        \leq &(1-s)\int_{\mathbb S^{n-2}}\int_0^{\delta}\rho_L\Big(\frac{(r\theta,f(r\theta))}{|(r\theta,f(r\theta))|}\Big)^{n+s}\frac{\kappa_K(o,\theta)}{2}r^{-s}drd\tH^{n-2}(\theta)+(1-s)o(1)\int
        _0^{\delta}r^{-s}dr.
    \end{align*}
    By \eqref{lem 1 appro 1}, we have
    \begin{align*}
        (1-s)\int_{D_1}&\rho_L(u)^{n+s}\rho_K(u)^{-s}du\\
        \leq &(1-s)\int_{\mathbb S^{n-2}}\int_0^{\delta}(\rho_L(\theta,0)^{n+1}+\epsilon)\frac{\kappa_K(o,\theta)}{2}r^{-s}drd\tH^{n-2}(\theta)+o(1)\delta^{1-s}\\
        =&\frac{\delta^{1-s}}{2}\int_{\mathbb S^{n-2}}(\rho_L(\theta,0)^{n+1}+\epsilon)\kappa_K(o,\theta)d\tH^{n-2}(\theta)+o(1)\delta^{1-s}.
    \end{align*}
    Let $s\rightarrow 1^{-}$. By the arbitrariness of $\epsilon$, we have
    \[\limsup_{s\rightarrow1^-}(1-s)\int_{D_1}\rho_L(u)^{n+s}\rho_K(u)^{-s}du\leq \frac{1}{2}\int_{\mathbb S^{n-2}}\rho_L(\theta,0)^{n+1}\kappa_K(o,\theta)d\tH^{n-2}(\theta)+o(1).\]

    For $u\in D\backslash D_1$ and $(y^{\prime},f(y^{\prime}))=\rho_K(u)u$, we have $\rho_K(u)\ge \delta $. Therefore $\rho_L(u)^{n+s}\rho_K(u)^{-s}$ is uniformly bounded, i.e., $\rho_L(u)^{n+s}\rho_K(u)^{-s}\leq c$ for some constant $c$ and for all $u\in D\backslash D_1$ and $s\in(0,1)$. Then
    \[\limsup_{s\rightarrow 1^{-}}(1-s)\int_{D\backslash D_1}\rho_L(u)^{n+s}\rho_K(u)^{-s}du\leq \limsup_{s\rightarrow 1^{-}}(1-s)\int_{D\backslash D_1} c du=0,\]
    which implies
    \[\limsup_{s\rightarrow 1^-}(1-s)\int_D\rho_L(u)^{n+s}\rho_K(u)^{-s}du\leq \frac{1}{2}\int_{\mathbb S^{n-2}}\rho_L(\theta,0)^{n+1}\kappa_K(o,\theta)d\tH^{n-2}(\theta)
    +o(1).\]
    Letting $\delta\rightarrow 0^+$, we then have
    \[\limsup_{s\rightarrow 1^-}(1-s)\int_D\rho_L(u)^{n+s}\rho_K(u)^{-s}du\leq \frac{1}{2}\int_{S^{n-2}}\rho_L(\theta,0)^{n+1}\kappa_K(o,\theta)d\tH^{n-2}(\theta).\]

    By the other side of the bound in \eqref{lem 1 appro 1} and Fatou's Lemma, we repeat the same arguments and get
    \[\liminf_{s\rightarrow 1^-}(1-s)\int_D\rho_L(u)^{n+s}\rho_K(u)^{-s}du\geq \frac{1}{2}\int_{S^{n-2}}\rho_L(\theta,0)^{n+1}\kappa_K(o,\theta)d\tH^{n-2}(\theta),\]
    which concludes the proof.
\end{proof}

\begin{lemma}\cite[Lemma 4.1]{XZ}\label{xz-lem}
    Suppose $K\in\tk^n$ has $C^2$ boundary and everywhere positive Gauss curvature. Let $z\in\partial K$ and $v$ be the normal vector of $K$ at $z$. For $u\in\sn\cap v^{\perp}$, we have
    \[\kappa_{K|u^{\perp}}(z|u^{\perp})=\frac{\kappa_K(z)}{\kappa_K(z,u)}\]
\end{lemma}

\begin{lemma}\label{conv1}
Suppose $K\in\tk^n$ has $C^2$ boundary and everywhere positive Gauss curvature. Let $L\in\tk_e^n$ and  $\eta\subset\sn$ be a Borel subset. Then we have
    %\[S_{n-2}(K,ZL,\eta)=\frac{2}{n^2-1}\int_{\sn}\rho_L(u)^{n+1}\int_{\eta\cap u^{\perp}}\frac{\kappa_K(z,u)}{\kappa_K(z)}d\tH^{n-2}(v)du\]
    \[S_{n-2}(K,ZL,\eta)=\frac{2}{n^2-1}\int_{\eta}\int_{\sn\cap v^{\perp}}\rho_L(u)^{n+1}\frac{\kappa_K(z,u)}{\kappa_K(z)}d\tH^{n-2}(u)dv,\]
    where $z=\nu_K^{-1}(z)\in\partial K$.
\end{lemma}

\begin{proof}
    Step 1: We first show that
    \begin{equation}\label{lem3 eq 3}
        S_{n-2}(K,ZL,\eta)=\frac{1}{n+1}\int_{\sn}S_{n-2}(K,[-u,u],\eta)\rho_L(u)^{n+1}du
    \end{equation}
    Since $C^2(\sn)$ is dense in $C(\sn)$ and every $f\in C^2(\sn)$ is the difference of two support functions, it suffices to show
    \[\int_{\sn}h_F(v)dS_{n-2}(K,ZL,v)=\frac{1}{n+1}\int_{\sn}\rho_L(u)^{n+1}\int_{\sn}h_F(v)dS_{n-2}(K,[-u,u],v)du,\]
    where $F\subset \rn$ is a nonempty convex compact set.

    Note that
    \[\int_{\sn}h_F(v)dS_{n-2}(K,ZL,v)=nV(F,K,\ldots, K, ZL),\]
    where the mixed volume $V(F,K,\ldots, K, ZL)$ is symmetric in its arguments. Then by the definition of $ZL$ and Fubini's theorem, we have
    \begin{equation}\label{lem3 eq1}
    \begin{aligned}
        \int_{\sn}h_F(v)dS_{n-2}(K,ZL,v)&=\int_{\sn}h_{ZL}(v)dS_{n-2}(K,F,v)\\
        &=\frac{1}{n+1}\int_{\sn}\int_{\sn}| v\cdot u |\rho_L(u)^{n+1}du dS_{n-2}(K,F,v)\\
        &=\frac{1}{n+1}\int_{\sn}\rho_L(u)^{n+1}\int_{\sn}| v\cdot u |dS_{n-2}(K,F,v)du.
    \end{aligned}
    \end{equation}
    Similarly, 
    \begin{equation}\label{lem3 eq2}
        \int_{\sn}| v\cdot u |dS_{n-2}(K,F,v)=V([-u,u], K,\ldots, K, F)=\int_{\sn}h_F(v)dS_{n-2}(K,[-u,u], v).
    \end{equation}
    Then by \eqref{lem3 eq1}, \eqref{lem3 eq2} and the Fubini's theorem, we have
    \[\int_{\sn}h_F(v)dS_{n-2}(K,ZL,v)=\frac{1}{n+1}\int_{\sn}\rho_L(u)^{n+1}\int_{\sn}h_F(v)dS_{n-2}(K,[-u,u],v)du,\]
    which gives \eqref{lem3 eq 3}.

    \noindent Step 2: We aim to show
    \[S_{n-2}(K,[-u,u],\eta)=\frac{2}{n-1}\int_{\eta\cap u^{\perp}}\frac{\kappa_K(z,\theta)}{\kappa_K(z)}d\tH^{n-2}(\theta).\]

    For a convex compact set $F\subset \rn$, by the Cauchy projection formula we have
    \begin{align*}
        \frac{1}{2}\int_{\sn}| v\cdot u | dS_{n-2}(K,F,v)&=v(K|u^{\perp},\ldots, K|u^{\perp}, F|u^{\perp})\\
        &=\frac{1}{n-1}\int_{\sn\cap u^{\perp}}h_F(v)dS^{\prime}(K|u^{\perp},v),
    \end{align*}
    where $v(K|u^{\perp},\ldots, K|u^{\perp}, F|u^{\perp})$ denotes the mixed volume in $u^{\perp}$ and  $S^{\prime}(K|u^{\perp},\cdot)$ denotes the surface area measure of $K|u^{\perp}$ in $u^{\perp}$. Recall that the density of the surface area measure is the reciprocal of the Gauss curvature, which implies
    \[\frac{1}{2}\int_{\sn}| v\cdot u | dS_{n-2}(K,F,v)=\frac{1}{n-1}\int_{\sn\cap u^{\perp}}\frac{h_F(v)}{\kappa_{K|u^{\perp}}(z|u^{\perp})}d\tH^{n-2}(v),\]
    where $z=\nu_K^{-1}(v)\in\partial K$.  
    Combining with Lemma \ref{xz-lem}, we have
    \[\int_{\sn}h_F(v)dS_{n-2}(K,[-u,u],v)=\frac{2}{n-1}\int_{\sn\cap u^{\perp}}h_F(v)\frac{\kappa_K(z,u)}{\kappa_{K}(z)}d\tH^{n-2}(v).\]

    Step 3: By step $1$ and step $2$, we have
    \[S_{n-2}(K,ZL,\eta)=\frac{2}{n^2-1}\int_{\sn}\int_{\eta\cap u^{\perp}}\rho_L(u)^{n+1}\frac{\kappa_K(z,u)}{\kappa_K(z)}d\tH^{n-2}(v)du.\]
    Since
    \[\int_{\sn}\int_{\sn\cap u^{\perp}}f(u,v)dvdu=\iint_{u\cdot v=0}f(u,v)d(u,v)=\int_{\sn}\int_{\sn\cap v^{\perp}}f(u,v)dudv,\]
    we obtain the desired result.
%    \begin{align*}
%        S_{n-2}(K,ZL,\eta)&=\frac{2}{n^2-1}\int_{\sn}\frac{1}{\kappa_K(z)}\int_{\eta\cap v^{\perp}}\rho_L(u)^{n+1}{\kappa_K(z,u)}d\tH^{n-2}(u)dv\\
%        &=\frac{2}{n^2-1}\int_{\sn}\int_{\eta\cap v^{\perp}}\rho_L(u)^{n+1}{\kappa_K(z,u)}d\tH^{n-2}(u)dS(K)
%    \end{align*}
\end{proof}
 
\begin{theorem}\label{lim 1} 
Suppose $K\in\tk^n$ has $C^2$ boundary and everywhere positive Gauss curvature and $L\in\tk_e^n$. Then
    \[(1-s)\mathcal{A}_s(K,L,\cdot)\rightarrow \frac{n^2-1}{2}S_{n-2}(K,ZL,\cdot),\]
as $s\rightarrow 1^-$.
\end{theorem}

\begin{proof}
    Let $(s_i)$ be a convergent sequence in $(0,1)$, where $s_i\rightarrow 1$ as $i\rightarrow \infty$. Assume $f\in C(\sn)$ is arbitrary and hence $f$ is bounded. That is,  $|f(v)|\leq c$ for all $v\in\sn$.
    
    Since $L\in\tk_e^n$, its radial function $\rho_L$ is bounded. Then there exists $l>0$ such that
    \[\tilde{V}_{n+s_i}(K,L,z)=\frac{1}{n}\int_{\mathbb S_z^+}\rho_{L}(u)^{n+s_i}X_K(z,u)^{-s_i}du\leq \frac{l}{n}\int_{\mathbb S_z^+}X_K(z,u)^{-s_i}du=l\tilde{V}_{n+s_i}(K,z),\]
    and therefore
    \begin{equation*}
        (1-s_i)f(\nu_K(z))\tilde{V}_{n+s_i}(K,L,z)\leq cl(1-s_i)\tilde{V}_{n+s_i}(K,z)
    \end{equation*}
    It was shown in \cite[Corollary 4.14]{LXYZ2020} that 
    \[\lim_{s\rightarrow 1^-}\int_{\partial K}(1-s)\tilde{V}_{n+s}(K,z)d\tH^{n-1}(z)=\int_{\partial K}\lim_{s\rightarrow 1^-}(1-s)\tilde{V}_{n+s}(K,z)d\tH^{n-1}(z).\]
    Then the dominated convergence theorem implies
    \begin{align*}
        \lim_{i\rightarrow\infty}(1-s_i)\int_{\sn}f(v)d{\mathcal A}_{s_i}(K,L,v)&=\lim_{i\rightarrow\infty}\frac{n}{s_i}\int_{\partial K}(1-s_i)f(\nu_{K}(z))\tilde{V}_{n+s_i}(K,L,z)d\tH^{n-1}(z)\\
        &=n\int_{\partial K}\lim_{i\rightarrow \infty} (1-s_i) f(\nu_{K}(z))\tilde{V}_{n+s_i}(K,L,z)d\tH^{n-1}(z).
    \end{align*}
    Combining the last equation with Lemma \ref{conv} and Lemma \ref{xz-lem}, we have
    \begin{align*}
        \lim_{s\rightarrow 1^-}(1-s)\int_{\sn}f(v)&d{\mathcal A}_{s_i}(K,L,v)\\
        &=\int_{\partial K}f(\nu_K(z))\int_{\sn\cap \nu_K(z)^{\perp}}\rho_L(\theta)^{n+1}\kappa_K(z,\theta)d\tH^{n-2}(\theta)d\tH^{n-1}(z)\\
        &=\int_{\sn}f(v)\int_{\sn\cap v^{\perp}}\rho_L(\theta)^{n+1}\frac{\kappa(z,\theta)}{\kappa_{K}(z)}d\tH^{n-2}(\theta)d\tH^{n-1}(z),
    \end{align*}
    where $z=\nu_K^{-1}(v)\in\partial K$. Hence we obtain
    \[\lim_{i\rightarrow \infty}(1-s_i)\int_{\sn}f(v)d{\mathcal A}_{s_i}(K,L,v)=\frac{n^2-1}{2}\int_{\sn}f(v)dS_{n-2}(K,ZL,v),\]
    which concludes the proof.
\end{proof}

\section{Minkowski problems for anisotropic fractional area measures}

The existence of solutions to the Minkowski problem is closely related to an optimization problem of a suitable functional under geometric constraints. In this section, we begin by considering the following optimization problem:

\begin{equation}\label{op}
    \underset{K\in\tk^n}{\inf}\Big\{\int_{\sn}h_K(v)d\mu(v): P_s(K,L)=1\Big\},
\end{equation}
where $\mu$ is a finite Borel measure on $\sn$.

\begin{theorem}\label{min-prob-ex}
    Let $\mu$ be a finite Borel measure on $\sn$. Suppose $L\in\tk_e^n$ and $s\in(0,1)$. If $K_0\in\tk^n$ satisfies $P_s(K_0,L)=1$ and is a minimizer to the optimization problem \eqref{op}, in the sense that
    \[\int_{\sn}h_{K_0}d\mu=\underset{K\in\tk^n}{\inf}\Big\{\int_{\sn}h_K(v)d\mu(v): P_s(K,L)=1\Big\},\]
    then there exists $c>0$ such that
    \[\mu=\tA(cK_0,L,\cdot).\]
\end{theorem}

\begin{proof}
Let $f\in C(\sn)$ be arbitrary and $h_t=h_{K_0}+tf$, where $|t|\leq \delta$ for some sufficiently small $\delta>0$ such that
\[K_t=\{x\in\rn:  x\cdot v  \leq h_t(v),~~\text{for all}~~v\in\sn\}\]
is a convex body. 

Denote by $\gamma(t)=P_s(K_t,L)^{-\frac{1}{n-s}}$ and $\overline{K}_t=\gamma(t)K_t$, which implies that $\gamma(0)=1$. Since $P_s(cK,L)=c^{n-s}P_s(K,L)$ for any $c>0$, we have $P_s(\overline{K}_t,L)=1$. Let
\[\phi(t)=\int_{\sn}\gamma(t)h_t(v)d\mu(v).\]
Note that $h_{\overline{K}_t}(v)=\gamma(t)h_{K_t}(v)\leq \gamma(t)h_t(v)$ for any $v\in\sn$. Since $K_0$ is the minimizer of the optimization problem \eqref{op}, we have
\[\phi(t)=\int_{\sn}\gamma(t)h_t(v)d\mu(v)\ge\int_{\sn}h_{\overline{K}_t}(v)d\mu(v)\ge\int_{\sn}h_{K_0}(v)d\mu(v)=\phi(0),\]
for all $t\in(-\delta,\delta)$.

Theorem \ref{var} implies that $\gamma(t)$ is differentiable at $t=0$ and
\[\gamma^{\prime}(0)=-\frac{2n}{n-s}\int_{\sn}f(v)d\tA(K_0,L,v)\]
Therefore, we obtain
\[0=\frac{d}{dt}\Big|_{t=0}\phi(t)=-\frac{2n}{n-s}\int_{\sn}f(v)d\tA(K_0,L,v)+\int_{\sn}f(v)d\mu(v).\]
Since $f\in C(\sn)$ is arbitrary, the desired result $\mu=\tA(cK_0,L,\cdot)$ follows directly, where  $c=(2n/n-s)^{\frac{1}{n-s}}$.
\end{proof}

%\begin{lemma}\label{bd Ps}
%    Let $s\in(0,1)$, $K\in\tk^n$ and $L\in\tk_e^n$
%\end{lemma}

The following lemma gives a sufficient condition for the existence of a solution of the optimization problem \eqref{op}.

\begin{lemma}\label{sol of op}
Let $s\in(0,1)$. Suppose $\mu$ is a finite Borel measure on $\sn$ and $L\in\tk_e^n$. If $\mu$ is not concentrated in any subsphere and
\[\int_{\sn}vd\mu(v)=o,\]
then there exists $K_0\in\tk^n$ such that $P_s(K_0,L)=1$ and $K_0$ is a minimizer to \eqref{min-prob-ex}.
\end{lemma}

\begin{proof}
    Suppose $(K_i)$ in $\tk^n$ is a minimizing sequence to \eqref{min-prob-ex}, where $P_s(K_i,L)=1$. Since the centroid of $\mu$ is the origin, we have 
    \[\int_{\sn}h_{K+x}(v)d\mu(v)=\int_{\sn}h_K(v)d\mu(v)+\int_{\sn} x\cdot v d\mu(v)=\int_{\sn}h_{K}(v)d\mu(v),\]
    for any $x\in\rn$. Together with the fact that $P_s(K+x,L)=P_s(K,L)$, we can assume $K_i\in\tk_o^n$ without loss of generality.

     Let $M_i=\max_{u\in\sn}\rho_{K_i}(u)=\rho_{K_i}(u_{k_i})$ for some $u_{k_i}\in\sn$. Then we have
    \[\int_{\sn}h_{K_i}(v)d\mu(v)\ge M_i\int_{\sn} \max\{ u_{k_i}\cdot v, 0\}d\mu(v). \]
    Since $\mu$ is not concentrated in any subsphere and the centroid of $\mu$ is the origin, $\mu$ is not concentrated in any closed hemispheres. Thus,
    \[u\mapsto \int_{\sn}\max\{ u\cdot v, 0\}d\mu(v)\]
    is a positive continuous function on $\sn$ and therefore has a uniform lower bound $c_0>0$ on $\sn$. Then we have
    \[\int_{\sn}h_{K_i}(v)d\mu(v)\ge M_i c_0.\]

    Note that $(K_i)$ is a minimizing sequence. Thus $M_i$ is uniformly bounded, and there exists $M > 0$ such that $K_i \subset MB^n$ for all $i$. By the Blaschke selection theorem, there exists a compact convex set $K_0 \subset \mathbb{R}^n$ such that, up to a subsequence, $K_i \to K_0$.

    It remains to show that $K_0$ has nonempty interior, i.e., $K_0$ is a convex body. Since $L\in\tk_e^n$, there is $c>0$ such that
    \[\frac{1}{c}P_s(K_i)\leq P_s(K_i,L)\leq cP_s(K_i),\]
    where $c$ only depends on $L$. 

    Recall from Section 2.4, by \eqref{Ps-int},
        \[P_s(K_i)=\frac{n\omega_n}{s(1-s)}\int_{{\rm Aff} (n,1)}|K_i\cap l|^{1-s}dl.\]
    Then by Jensen's inequality, we have
    \[P_s(K_i)\leq \frac{n\omega_n}{s(1-s)}\Big(\int_{{\rm Aff}(n,1)}|K_i\cap l|dl\Big)^{1-s}=\frac{n\omega_n}{s(1-s)}V(K_i)^{1-s}.\]
    Therefore,
    \[1=P_s(K_i,L)\leq \frac{cn\omega_n}{s(1-s)}V(K_i)^{1-s}.\]
    By the continuity of $V$, we obtain $V(K_0)>0$, which implies that $K_0\in \tk^n$.
\end{proof}

We are now in the position to prove Theorem 1.1.
\begin{proof}[Proof of Theorem 1.1] The necessity is provided by Lemma \ref{centroid}, where
\[P_s(K,L)=\frac{2n}{n-s}\int_{\sn}h_K(v)d\tA(K,L,v)>0\]
implies that $\tA(K,L,\cdot)$ is not concentrated in any subsphere.

Let us now assume that $\mu$ is not concentrated in any subsphere and with the centroid at the origin. By Lemma \ref{sol of op}, we have the existence of $K\in\tk^n$ such that $\mu=\tA(K,L,\cdot)$.
\end{proof}

\section{Anisotropic fractional isoperimetric inequalities }

In the last section, we derive a necessary condition for the convexity of the optimizer in the anisotropic fractional isoperimetric inequality
\[P_s(E,L)\ge \gamma_s(L)|E|^{\frac{n-s}{n}},\]
using the variational formula \eqref{var}.

\begin{theorem}
    Let $s\in(0,1)$ and $L\in\tk_e^n$. If $K_0$ is a convex body in $\rn$ such that
\[P_s(K_0,L)=\gamma_s(L)|K_0|^{\frac{n-s}{n}},\]
then there exists $c>0$ such that
\[\tA(K_0,L,\cdot)=cS_{n-1}(K_0,\cdot),\]
and therefore, $\tilde{V}_{n+s}(K_0,L,z)$ is constant for almost every $z\in\partial K_0$.
\end{theorem}

\begin{proof}
    Suppose $K_0\in\tk^n$ is an optimizer in the anisotropic fractional isoperimetric inequality. Let $f\in C(\sn)$ be arbitrary. We denote the associated Wulff shape by
    \[K_t=\{x\in\rn:  x\cdot v \leq h_{K_0}(v)+tf(v),~~\text{for all}~~v\in\sn\},\]
    where $|t|<\delta$ for some sufficiently small $\delta>0$ such that $K_t$ is a convex body. 
    
    Consequently, we have
    \[P_s(K_t,L)\ge \gamma_s(L)|K_t|^{\frac{n-s}{n}},\]
    which implies that the function
    \[\psi(t)=P_s(K_t,L)/|K_t|^{\frac{n-s}{n}}\]
    attains its minimum at $t=0$. By Theorem \ref{var} and Aleksandrov’s variational formula \eqref{ale var} for the volume functional,
    \[0=\frac{d}{dt}\Big|_{t=0}\psi(t)=\frac{2n}{|K_0|^{\frac{n-s}{n}}}\int_{\sn}f(v)d\tA(K_0,L,v)-\frac{(n-s)P_s(K_0,L)}{n|K_0|^{2-\frac{s}{n}}}\int_{\sn}f(v)dS_{n-1}(K_0,v).\]
    By the arbitrariness of $f$, we obtain
    \[S_{n-1}(K_0,\cdot)=c\tA(K_0,L,\cdot),\]
    where $c=\frac{2n^2|K_0|}{(n-s)P_s(K_0,L)}$.

    Moreover, by comparing the definition of $\tA(K_0,L,\cdot)$ with $S_{n-1}(K_0,\cdot)$, where
    \[\tA(K_0,L,\eta)=\frac{n}{s}\int_{\nu_K^{-1}(\eta)}\tilde{V}_{n+s}(K_0,L,z)d\tH^{n-1}(z),~~S_{n-1}(K)=\int_{\nu_K^{-1}(\eta)}d\tH^{n-1}(z),\]
    we have
    \[\tilde{V}_{n+s}(K_0,L,z)=\frac{(n-s)P_s(K_0,L)}{2n^2|K_0|}=\frac{n-s}{2n^2}\gamma_s(L)|K_0|^{-\frac{s}{n}},\]
    for almost every $z\in\partial K$.
\end{proof}

\noindent{\bf Acknowledgements:} 
This research was funded in whole or in part by the Austrian Science Fund (FWF) doi/10.55776/37030. For open access purposes, the author has applied a CC BY public copyright license to any author accepted manuscript version arising from this submission.

\bibliographystyle{abbrv}

\end{document}